\documentclass[a4paper,11pt]{amsart}

\usepackage[left=2cm,top=3cm,bottom=2.5cm,right=2.5cm]{geometry}

\usepackage{amsthm}
\usepackage{amssymb}
\usepackage{amsmath}
\usepackage{mathtools}
\usepackage{pdfpages}
\usepackage{changes}
\usepackage{exscale,cite,color,amsopn}
\usepackage[colorlinks=true,pdftex,unicode=true,linktocpage,bookmarksopen,hypertexnames=true]{hyperref}
\usepackage{aliascnt}
\usepackage{csquotes}
\usepackage{colortbl}
\usepackage{enumitem}
\usepackage{verbatim}
\usepackage{hyperref}
\usepackage{tikz-cd}
\usepackage[capitalise]{cleveref}

\usetikzlibrary{trees}
\usetikzlibrary{calc}

\renewcommand{\leq}{\leqslant}
\renewcommand{\geq}{\geqslant}
\DeclareMathOperator{\id}{id}

\DeclareMathOperator{\Soc}{Soc}

\newcommand{\Z}{\mathbb{Z}}

\newcommand{\Zen}{\mathcal{Z}}
\newcommand{\Aut}{\operatorname{Aut}}

\newcommand{\Hol}{\operatorname{Hol}}

\newcommand{\Sym}{\mathfrak{S}}

\newcommand{\At}{\mathrm{At}}
\newcommand{\Atd}{\mathrm{At}^{\ast}}

\newcommand{\IGam}{\mathrm{I}_{\Gamma}}
\newcommand{\IG}{\mathrm{I}_G}

\newcommand{\gen}[1]{\left\langle #1 \right\rangle}
\newcommand{\genrel}[2]{\left\langle \ #1 \ \vline \ #2 \ \right\rangle}

\newcommand{\subcirc}[1]{\underset{#1}{\circ}}

\makeatletter
\numberwithin{equation}{section}
\numberwithin{figure}{section}
\numberwithin{table}{section}
\newtheorem{thm}{Theorem}[section]
\newtheorem*{thm*}{Theorem}
\newtheorem{lem}[thm]{Lemma}

\newtheorem{pro}[thm]{Proposition}

\theoremstyle{definition}
\newtheorem{defn}[thm]{Definition}
\newtheorem*{defn*}{Definition}

\newtheorem*{question*}{Question}

\newtheorem*{convention*}{Convention}
\newtheorem{rem}[thm]{Remark}

\makeatother

\title{Groups of $\IG$-Type}
\author{Carsten Dietzel}
\date{\today}

\address[Carsten Dietzel]{Normandie Univ, UNICAEN, CNRS, LMNO, 14000 Caen, France}
\email{carsten.dietzel@unicaen.fr}

\begin{document}

\begin{abstract}
In this work, we address a question posed by Dehornoy et al. in the book \textit{Foundations of Garside Theory} that asks for a theory of groups of \textit{$\mathrm{I}_G$-type} when $G$ is a Garside group. In this article, we introduce a broader notion than the one suggested by Dehornoy et al.: given a left-ordered group $G$, we define a \textit{group of $\mathrm{I}_G$-type} as a left-ordered group whose partial order is isomorphic to those of $G$. Furthermore, we develop methods to give a characterization of groups of $\mathrm{I}_{\Gamma}$-type in terms of \textit{skew braces} when $\Gamma$ is an Artin-Tits group of spherical type and classify all groups of $\mathrm{I}_{\Gamma}$-type where $\Gamma$ is an irreducible spherical Artin-Tits group, therefore providing an answer to another question of Dehornoy et al. concerning $\mathrm{I}_{B_n}$ structures where $B_n$ is the braid group on $n$ strands with its canonical Garside structure.
\end{abstract}

\maketitle

\section*{Introduction}

The goal of this article is to set up an order-theoretical framework in order to provide an answer to the following Garside-theoretical question that was posed in the book \emph{Foundations of Garside Theory} by Dehornoy et al. \cite[Question 38]{Foundations_Garside}:

\begin{question*}
    Can one characterize the Garside groups that admit an $\IGam$-structure?
\end{question*}

Recall \cite[Chapter I]{Foundations_Garside} that a \emph{Garside group} is a group $G$ with a distinguished submonoid $G^+ \subseteq G$ - its \emph{Garside monoid} - such that
\begin{enumerate}
    \item $G$ is a left- and a right group of fractions for the monoid $G^+$, that is, $G = \{g^{-1}h : g,h \in G^+ \} = \{gh^{-1} : g,h \in G^+ \}$.
    \item There is a map $\nu: G^+ \to \Z_{\geq 0}$ such that $\nu(g) = 0 \Leftrightarrow g = e$ and $\nu(gh) \geq \nu(g) + \nu(h)$. for all $g,h \in G^+$.
    \item $G^+$ is a lattice with respect to left-divisibility, also $G^+$ is a lattice with respect to right-divisibility.
    \item There is a distinguished element $\Delta \in G^+$ - the \emph{Garside element} - such that the left- and right-divisors of $\Delta$ in $G^+$ coincide, form a finite set and generate $G^+$ as a monoid.
\end{enumerate}

This definition - although named after Garside - has only slowly emerged from Garside's solution of the conjugacy problem on braid groups \cite{Garside_braid}: after Garside's discovery of what is now known as a Garside structure on the braid groups, similar structures have first been discovered by Brieskorn and Saito \cite{ABS_Coxeter} in the slightly more general case of spherical Artin-Tits groups, leading to the solution of the conjugacy problem for these groups, before the notion of a Garside group was introduced in its final form by Dehornoy \cite{Dehornoy_Garside_groups}, in order to describe a broader class of groups that admit efficient solutions to difficult algorithmic problems, namely the word- and the conjugacy problem. 

Note that the axioms of a Garside group guarantee that $G^+$ is generated, as a monoid, by the set $\At(G)$ of \emph{atoms} in $G^+$, that is, those elements in $G^+ \setminus \{ e \}$ that have no proper divisors in $G^+$. As a consequence, $G$ is generated by $\At(G)$ as a group. Therefore, it makes sense to consider the \emph{directed} Cayley graph of a Garside group with respect to $\At(G)$, that is, the directed graph with vertex set $G$ where $g,h$ are connected by an edge $g \to h$ if and only if there is an $x \in \At(G)$ such that $gx = h$. A group of $\IG$-type is then defined by Dehornoy et al. \cite[pp.602-603]{Foundations_Garside} for a Garside group $G$ as a Garside group $H$ such that $H$ and $G$ have isomorphic directed Cayley graphs. We remark here that this is essentially an order-theoretic property, as the left-divisibility order on a Garside group determines the directed Cayley graph. Therefore, a Garside group of $\IG$-type is one whose left-divisibility order is isomorphic to the left-divisibility order of $G$.

The question of Dehornoy et al. is motivated by the work of Gateva-Ivanova and Van den Bergh \cite{GIVdB_IType} and Jespers and Okniński \cite{JO_IType} on groups of \emph{$I$-type} which are, according to the definition of Dehornoy et al., the groups of $\mathrm{I}_{\Z^n}$ type, where $\Z^n$ is considered as a Garside group with Garside monoid $(\Z^n)^+ = \Z_{\geq 0}^n$. They have shown that each group of $I$-type is a regular subgroup of the group $\Z^n \rtimes \Sym_n$ acting affinely on $\Z^n$, where the symmetric group $\Sym_n$ acts by permutations of coordinates on $\Z^n$. Furthermore, due to the mentioned work of Gateva-Ivanova and Van den Bergh, and Jespers and Okniński, together with work of Chouraqui \cite{Chouraqui_Garside}, it is now known that Garside groups of $I$-type coincide with structure groups of finite involutive, nondegenerate set-theoretic solutions to the Yang-Baxter equation.

In order to motivate the investigation of \emph{Yang-Baxter like} structures for braid groups, Dehornoy et al. ask for a characterization of groups of $\mathrm{I}_{B_n}$-type where $B_n$ is the braid group on $n$ strands, together with its usual Garside structure (see \cref{sec:prelim}). In this article, we aim to give a full solution to this problem.

In \cref{sec:order_automorphisms}, we define the notion of $\IG$-type in the more general case when $G$ is a \emph{left-ordered group}, that is, a group with a left-invariant partial order, and define a group of \emph{$\IG$-type} as a left-ordered group $H$ with an order-isomorphism $\iota: (H,\leq) \overset{\sim}{\to} (G,\leq)$. These data $(H,\iota)$ will be conceptualized by the notion of an \emph{$\IG$-formation}. We will prove that if $(G,\leq)$ is a lattice with finitely many atoms that satisfies certain \emph{rigidity} conditions, the automorphism group of the ordered set $(G,\leq)$ contains $L_G$, the group of left translations of $G$, as a finite index subgroup. As a consequence, these groups admit only finitely many $\IG$-formations, up to equivalence.

We will proceed in \cref{sec:IG_for_AT} with a description of $\IGam$-formations whenever $\Gamma$ is a spherical Artin-Tits group. We will show that the automorphism group of the ordered set $(\Gamma,\leq)$ is a semidirect product of $L_{\Gamma}$ and $\mathfrak{D}_{\Gamma}$, the group of \emph{diagram automorphisms}, which is an analogue of a result of Björner \cite[Theorem 3.2.5]{bjorner_combinatorics} on automorphisms of weak order. As a consequence, we can prove that $\IGam$-formations for an Artin-Tits group $\Gamma$ are equivalent to certain \emph{skew brace} structures on $\Gamma$ \cite{Guarnieri_Vendramin} and use this result to classify all non-trivial $\IGam$-formations whenever $\Gamma$ is an \emph{irreducible} spherical Artin-Tits group.

In particular, we will answer the question of Dehornoy et al. by showing that the braid group $B_n$ ($n \geq 3$) admits exactly one non-trivial $\mathrm{I}_{B_n}$-formation.

Note that our notion of $\IG$-type is only remotely related to the \emph{monoids of $IG$-type} introduced by Goffa and Jespers \cite{goffa_jespers}, that are brace-like structures on abelian monoids.

\section{Preliminaries} \label{sec:prelim}

\subsection{Left-ordered groups}

Recall that a \emph{lattice} is a partially ordered set $L = (L,\leq)$ such that for any $x,y \in L$, the binary \emph{join}
\[
x \vee y = \min \{ z \in L \ : \ (z \geq x) \ \& \ (z \geq y) \}
\]
and dually, the binary \emph{meet}
\[
x \wedge y = \max \{ z \in L \ : \ (z \leq x) \ \& \ (z \leq y) \}
\]
exist. A \emph{left-ordered group} is a pair $G = (G,\leq)$ where $G$ is a group and $\leq$ is a partial order that is \emph{left-invariant} in the sense that $y \leq z$ implies $xy \leq xz$ for all $x,y,z \in G$. If a left-ordered group $G$ is a lattice under its partial order, one says that $G$ is a \emph{left $\ell$-group}.

In order to reference it later, we recall the following elementary fact from \cite{Rump-Geometric-Garside}:

\begin{pro} \label{pro:left_l_groups_are_torsion_free}
    If $G$ is a left $\ell$-group, then $G$ is torsion-free.
\end{pro}

\begin{proof}
    If $g \in G$ and $n > 0$ are such that $g^n = e$, then
    \[
    g \cdot \left( \bigvee_{k=0}^{n-1} g^k \right) = \bigvee_{k=0}^{n-1} g^{k+1} = \bigvee_{k=0}^{n-1} g^k,
    \]
    so $g = e$.
\end{proof}

Two distinguished subsets of any left-ordered group $G$ are its \emph{positive cone} $G^+ = \{g \in G: g \geq e \}$ and its \emph{negative cone} $G^- = \{g \in G: g \leq e \}$. Note that $G^+$ and $G^-$ determine each other, that is, $G^- = (G^+)^{-1}$. Furthermore, the order of any left-ordered group is determined by its positive or negative cone, as
\[
y^{-1}x \in G^- \Leftrightarrow x \leq y \Leftrightarrow x^{-1}y \in G^+.
\]
A left-ordered group $G$ is called \emph{noetherian} whenever each ascending chain $x_1 \leq x_2 \leq \ldots$ in $G^-$ becomes stationary and each descending chain $y_1 \geq y_2 \geq \ldots$ in $G^+$ becomes stationary. By left-invariance, this implies that each ascending chain that is bounded from above, becomes stationary, and each descending chain that is bounded from below, becomes stationary.

Furthermore, we denote by $G^{\mathrm{op}}$ the left-ordered group that is obtained by turning over the order of $G$, that is $g \leq_{\mathrm{op}} h \Leftrightarrow h \leq g$. It is readily checked that this is indeed a left-ordered group with $(G^{\mathrm{op}})^+ = G^-$, and that $G^{\mathrm{op}}$ is a left $\ell$-group if and only if $G$ is a left $\ell$-group.

If $P$ is a partially ordered set and $x,y \in P$ we write $x \prec y$ if $x < y$ and there is no $z \in P$ such that $x < z < y$. Given $x,y \in P$ with $x \leq y$, we call a finite sequence $(x_i)_{0 \leq i \leq l}$ for some $l \in \Z_{\geq 0}$ a \emph{successor chain} of \emph{length} $l$ from $x$ to $y$ when $x = x_0 \prec x_1 \prec \ldots \prec x_l = y$. Given two elements $x,y \in P$ with $x \leq y$, we define their \emph{relative height} $H(x,y)$ as the minimal $l$ such that there is a successor chain $(x_i)_{0 \leq i \leq l}$ from $x$ to $y$, if such a chain exists. If no successor chain exists, we define $H(x,y) = +\infty$.

\begin{pro} \label{pro:successor_chains_exist}
    If $G$ is a noetherian left-ordered group, then for any $x,y \in G$ with $x \leq y$, we have $H(x,y) < +\infty$.
\end{pro}

\begin{proof}
    We have to show the existence of a successor chain from $x$ to $y$. By noetherianity, there cannot be an infinite chain $y = y_0 > y_1 > \ldots$ where $y_i > x$ for all $i \geq 0$ which shows the existence of an $x_1 \in G$ with $y \geq x_1 \succ x_0 = x$. Therefore, given a successor chain $x = x_0 \prec x_1 \prec \ldots \prec x_i$ $(i \geq 0)$, where $x_i < y$, one can find $x_{i+1}$ with $x_i \prec x_{i+1} \leq y$. This process terminates when $x_i = y$, which is guaranteed by noetherianity. 
\end{proof}

Given a noetherian left-ordered group, we define the (\emph{absolute}) \emph{height} of $g \in G^+$ as the quantity $H(g) = H(e,g)$. By \cref{pro:successor_chains_exist}, $H(g)$ is always finite. 

In a left-ordered group $G$, we call an element $x \in G$ an \emph{atom} if $g \succ e$, and a \emph{dual atom} if $x \prec e$. We write $\At(G)$ for the set of atoms and $\Atd(G)$ for the set of dual atoms in $G$.

\begin{pro} \label{pro:height_is_length_of_minimal_factorization}
    If $G$ is a noetherian left-ordered group and $g \in G^+$, then $H = H(g)$ is the minimal integer $H \geq 0$ such that there is a factorization $g = x_1x_2 \ldots x_H$ with $x_i \in \At(G)$ ($1 \leq i \leq H$).
\end{pro}

\begin{proof}
    Given a successor chain $(g_i)_{1 \leq i \leq H}$ from $e$ to $g$, the elements $x_i = g_{i-1}^{-1}g_i \in \At(G)$ ($1 \leq i \leq H$) constitute a factorization of $g$, that is, $g = x_1 x_2 \ldots x_H$, which has $H$ factors. On the other hand, given a factorization $g = x_1 x_2 \ldots x_H$ with $x_i \in \At(G)$ ($1 \leq i \leq H$), there is a successor chain $(g_i)_{0 \leq i \leq H}$ of length $H$ from $e$ to $g$ whose elements are given by $g_0 = e$ and $g_i = g_{i-1}x_i$ ($1 \leq i \leq H$).
    \end{proof}

\begin{pro} \label{pro:noetherian_l_groups_generated_by_atoms}
    If $G$ is a noetherian left $\ell$-group, then $G^+ = \gen{\At(G)}_{\mathrm{mon}}$ and $G = \gen{\At(G)}_{\mathrm{gr}}$.
\end{pro}

\begin{proof}
    By \cref{pro:height_is_length_of_minimal_factorization}, $G^+ \subseteq \gen{\At(G)}_{\mathrm{mon}} \subseteq G^+$. For arbitrary $g \in G$, we can write $g = g_1^{-1}g_2$ where $g_1 = g^{-1}(g \vee e) \in G^+$ and $g_2 = g \vee e \in G^+$, so $G = \gen{G^+}_{\mathrm{gr}} \subseteq \gen{\At(G)}_{\mathrm{gr}} \subseteq G$.
\end{proof}

\subsection{Spherical Artin-Tits groups}

Here, we recapitulate part of the theory of Artin-Tits groups.

Recall that a \emph{Coxeter matrix} is given by a mapping $m: S \times S \to \Z_{\geq 1}$; $(i,j) \mapsto m_{ij}$ on some set $S$, such that $m_{ij} = m_{ji}$ for all $i,j \in S$ and $m_{ij} = 1$ if and only if $i = j$.

Given a monoid $M$, an integer $k \geq 1$ and elements $x,y \in M$, we denote the corresponding \emph{braid term} by
\begin{equation*} \label{def:braid_term}
    r_k(x,y) = \begin{cases}
    (xy)^l & k = 2l, \\
    (xy)^lx & k = 2l+1,
\end{cases}
\end{equation*}

Given a Coxeter matrix $m: S \times S \to \Z_{\geq 0}$, the corresponding \emph{Artin-Tits monoid} is defined by generators and relations as
\[
\Gamma^+_{m} = \genrel{\sigma_i, \ i \in S}{r_{m_{ij}}(\sigma_i,\sigma_j) = r_{m_{ij}}(\sigma_j,\sigma_i), \ i,j \in S}_{\mathrm{mon}}.
\]
Similarly, the \emph{Artin-Tits group} $\Gamma_m$ is the group defined by the same generators and relations.

Furthermore, recall that the \emph{Coxeter group} associated to a Coxeter matrix $m$ is defined by generators and relations as
\[
G_m = \genrel{\sigma_i, \ i \in S}{r_{m_{ij}}(\sigma_i,\sigma_j) = r_{m_{ij}}(\sigma_j,\sigma_i), \ i,j \in S; \ \sigma_i^2 = 1, \ i \in S}_{\mathrm{gr}}.
\]
An Artin-Tits monoid $\Gamma_m^+$ resp. Artin-Tits group $\Gamma_m$ is called \emph{spherical} whenever $G_m$ is a finite group.

In the following, we will typically drop the subscript-$m$ for Artin-Tits monoids $\Gamma_m^+$ resp. -groups $\Gamma_m$ if the Coxeter matrix is clear from the context.

Recall the representation of Coxeter matrices by Coxeter graphs: given a Coxeter matrix $m$, the corresponding Coxeter graph is the labelled, undirected graph on the set $S$ where $i,j \in S$ are connected by an edge if and only if $m_{ij} \geq 3$. Furthermore, the edge $\{i,j \}$ is labelled by the quantity $m_{ij}$, which is dropped if $m_{ij} = 3$.

Let $m: S \times S \to \Z_{\geq 1}$ be a Coxeter matrix for the Artin-Tits monoid $\Gamma^+$ resp. -group $\Gamma$, and let $\phi \in \Sym_S$ be a permutation such that $m_{\phi(i)\phi(j)} = m_{ij}$ for all $i,j \in S$, then there is a unique automorphism $\delta_{\phi}$ of $\Gamma^+$ resp. $\Gamma$ such that $\delta_{\phi}(\sigma_i) = \sigma_{\phi(i)}$. An automorphism of the form $\delta_{\phi}$ is called a \emph{diagram automorphism}, and we denote by $\mathfrak{D}_{\Gamma^+}$ and $\mathfrak{D}_{\Gamma}$ the respective groups of diagram automorphisms.

Recall that a Artin-Tits group resp. Coxeter group is called \emph{irreducible} if and only if its Coxeter graph is connected. More precisely, this means that for any $x,y \in S$, there is a sequence $(x_i)_{1 \leq i \leq k}$ with $x_i \in S$ for some integer $k \geq 0$, such that $x = x_0$, $y = x_k$ and $m_{x_{i},x_{i+1}} \neq 2$ for $0 \leq i < k$.

The Coxeter matrices resp. graphs for spherical Artin-Tits groups are known. In the following, we only list, for future reference, the Coxeter graphs of irreducible spherical Artin-Tits groups where $\mathfrak{D}_{\Gamma}$ is non-trivial:
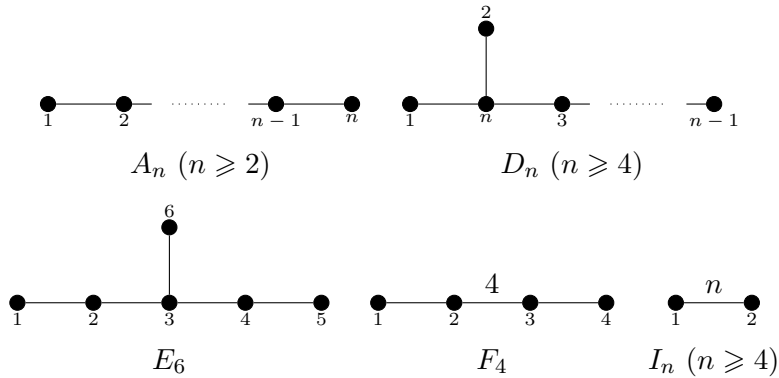
\begin{center} \label{tab:dynkin}
    \begin{tabular}{c c}
\begin{tikzpicture}
    \node (1) at (0,0) [draw, shape=circle, fill=black,inner sep=2pt] {} ;
    \node (2) at (1,0) [draw, shape=circle, fill=black,inner sep=2pt] {} ;
    \node (3) at (1.5,0) {} ;
    \node (4) at (2.5,0) {} ;
    \node (5) at (3,0) [draw, shape=circle, fill=black,inner sep=2pt] {} ;
    \node (6) at (4,0) [draw, shape=circle, fill=black,inner sep=2pt] {} ;
    \draw (1) -- (2) -- (3) {} ;
    \draw (4) -- (5) -- (6) {} ;
    \draw (3) -- (4) [dotted] {} ;
    \node at (1) [anchor=north] {\tiny{$1$}} ;
    \node at (2) [anchor=north] {\tiny{$2$}} ;
    \node at (5) [anchor=north] {\tiny{$n-1$}} ;
    \node at (6) [anchor=north] {\tiny{$n$}} ;
    \node at (3) [anchor=south,color = white] {$1$} ;
    \end{tikzpicture} &  \begin{tikzpicture}
    \node (1) at (0,0) [draw, shape=circle, fill=black,inner sep=2pt] {} ;
    \node (2) at (1,1) [draw, shape=circle, fill=black,inner sep=2pt] {} ;
    \node (n) at (1,0) [draw, shape=circle, fill=black,inner sep=2pt] {} ;
    \node (3) at (2,0) [draw, shape=circle, fill=black,inner sep=2pt] {} ;
    \node (4) at (2.5,0) {} ;
    \node (5) at (3.5,0) {} ;
    \node (n-1) at (4,0) [draw, shape=circle, fill=black,inner sep=2pt] {} ; ;
    \node at (1) [anchor = north] {\tiny{$1$}} ;
    \node at (2) [anchor = south] {\tiny{$2$}} ;
    \node at (3) [anchor = north] {\tiny{$3$}} ;
    \node at (n) [anchor = north] {\tiny{$n$}} ;
    \node at (n-1) [anchor = north] {\tiny{$n-1$}} ;
    \draw (1) -- (n) -- (3) ;
    \draw (n) -- (2) ;
    \draw (3) -- (4) ;
    \draw (4) -- (5) [dotted] {} ;
    \draw (5) -- (n-1) ;
    \end{tikzpicture}  \\
    $A_n$ ($n \geq 2$) & $D_n$ ($n \geq 4$) \\
\end{tabular}
\end{center}

\begin{center}
\begin{tabular}{c c c}
    \begin{tikzpicture}
    \node (1) at (0,0) [draw, shape=circle, fill=black,inner sep=2pt] {} ;
    \node (2) at (1,0) [draw, shape=circle, fill=black,inner sep=2pt] {} ;
    \node (3) at (2,0) [draw, shape=circle, fill=black,inner sep=2pt] {} ;
    \node (4) at (3,0) [draw, shape=circle, fill=black,inner sep=2pt] {} ;
    \node (5) at (4,0) [draw, shape=circle, fill=black,inner sep=2pt] {} ;
    \node (6) at (2,1) [draw, shape=circle, fill=black,inner sep=2pt] {} ;
    \draw (1) -- (2) -- (3) -- (4) -- (5) ;
    \draw (3) -- (6) ;
    \node at (1) [anchor=north] {\tiny{$1$}} ;
    \node at (2) [anchor=north] {\tiny{$2$}} ;
    \node at (3) [anchor=north] {\tiny{$3$}} ;
    \node at (4) [anchor=north] {\tiny{$4$}} ;
    \node at (5) [anchor=north] {\tiny{$5$}} ;
    \node at (6) [anchor=south] {\tiny{$6$}} ;
    \end{tikzpicture} & \begin{tikzpicture}
    \node (1) at (0,0) [draw, shape=circle, fill=black,inner sep=2pt] {} ;
    \node (2) at (1,0) [draw, shape=circle, fill=black,inner sep=2pt] {} ;
    \node (3) at (2,0) [draw, shape=circle, fill=black,inner sep=2pt] {} ;
    \node (4) at (3,0) [draw, shape=circle, fill=black,inner sep=2pt] {} ;
    \draw (1) -- (2) -- node [above] {$4$} (3) -- (4);
    \node at (1) [anchor=north] {\tiny{$1$}} ;
    \node at (2) [anchor=north] {\tiny{$2$}} ;
    \node at (3) [anchor=north] {\tiny{$3$}} ;
    \node at (4) [anchor=north] {\tiny{$4$}} ;
    \end{tikzpicture} & \begin{tikzpicture}
    \node (1) at (0,0) [draw, shape=circle, fill=black,inner sep=2pt] {} ;
    \node (2) at (1,0) [draw, shape=circle, fill=black,inner sep=2pt] {} ;
    \draw (1) -- node [above] {$n$} (2) ;
    \node at (1) [anchor=north] {\tiny{$1$}} ;
    \node at (2) [anchor=north] {\tiny{$2$}} ;
    \end{tikzpicture} \\
    $E_6$ & $F_4$ & $I_n$ ($n \geq 4$) \\
\end{tabular}
\end{center}

Note that we decided to put $G_2 = I_6$ and $H_2 = I_5$ here. Furthermore, we want to remark here that the generators of the Artin-Tits groups will in future calculations be numbered according to the labels of the vertices in the listed Coxeter graphs.

For spherical Artin-Tits groups, we have the following fundamental result by Brieskorn and Saito:

\begin{thm} \label{thm:brieskorn_saito}
    Let $\Gamma$ be a spherical Artin-Tits group. Then the canonical monoid homomorphism $\varepsilon: \Gamma^+ \to \Gamma$ identifies $\Gamma^+$ with the positive cone of a left-invariant, noetherian lattice order on $\Gamma$.
\end{thm}

\begin{proof}
    See \cite[Proposition 5.5, Satz 5.6]{ABS_Coxeter}.
\end{proof}

When talking about a spherical Artin-Tits group as a left $\ell$-group, we will always mean the lattice order defined by the positive cone $\Gamma^+$. Furthermore, by \cref{thm:brieskorn_saito}, we can from now on identify $\Gamma^+$ with the submonoid of $\Gamma$ generated by $\sigma_i$ ($i \in S$).

For all $\delta \in \mathfrak{D}(\Gamma)$, we also have $\delta(\Gamma^+) = \Gamma^+$, therefore we obtain:

\begin{pro} \label{pro:diagram_automorphisms_are_order_automorphisms}
    Let $\Gamma$ be a spherical Artin-Tits group, then $\mathfrak{D}_{\Gamma}$ is a group of automomorphisms of the ordered set $(\Gamma,\leq)$.
\end{pro}

\subsection{Skew braces}

\begin{defn} \label{defn:skew_brace}
    A \emph{skew brace} is a triple $B = (B,+,\circ)$ where $B$ is a set with two group operations $+$ and $\circ$ - both not necessarily commutative - that satisfy the identity
    \begin{equation}
        a \circ (b + c) = a \circ b - a + a \circ c.
    \end{equation}
\end{defn}

Note that $(B,+)$ and $(B,\circ)$ share the same identity!

A skew brace is called \emph{trivial} if the operations $+$ and $\circ$ coincide.

\emph{Sub-skew braces} of skew braces are, as usual, defined as subsets that are skew braces by restriction. Also, homomorphisms between skew braces are defined as maps respecting the skew brace operations.

Given a skew brace $B$, the \emph{$\lambda$-action} is the map
\[
\lambda: B \times B \to B ; \ (g,h) \mapsto \lambda_g(h) = - g + g \circ h.
\]
This map can be shown to satisfy the identities
\[
    \lambda_g(h_1 + h_2) = \lambda_g(h_1) + \lambda_g(h); \quad
    \lambda_{g_1 \circ g_2}(h) = \lambda_{g_1} (\lambda_{g_2}(h)), 
\]
so the assignment $(B,\circ) \to \Aut(B,+)$; $g \mapsto \lambda_g$ is a well-defined group homomorphism.

Note that a skew brace is trivial if and only if the $\lambda$-action is trivial!

Given a skew brace $B$, a subgroup $I \leq (B,+)$ is called a \emph{left ideal} if $\lambda_g(I) = I$ for all $g \in B$. Note that each left ideal is a subbrace of $B$ as $g \circ h = g + \lambda_g(h) \in I$ for $g,h \in I$. If, on top of that, $I$ is normal in $(B,+)$, we say $I$ is a \emph{strong left ideal}. Furthermore, a strong left ideal  $I$ is called an \emph{ideal} if $I$ is also normal in $(B,\circ)$.

Given a skew brace $B$ and an ideal $I \subseteq B$, the multiplicative and additive cosets of $I$ in $B$ coincide, and there is a well-defined skew brace structure on $B/I = \{b + I : b \in B\}$ that is given by $(a + I) + (b + I) = (a + b) + I$ and $(a + I) \circ (b + I) = (a \circ b) + I$. If the ideal is clear from the context, we abbreviate $a + I = \Bar{a}$.

A distinguished ideal of a skew brace $B$ is its \emph{socle}
\[
\Soc(B) = \ker(\lambda) = \{ g \in B \ : \ \forall h \in B:  \lambda_g(h) = h \} = \{ g \in B \ : \ \forall h \in B:  g \circ h = g + h \}.
\]

Given a skew brace $B$, one iteratively defines the \emph{retractions} $B_k$ ($k \geq 0$) by $B^{(0)} = B$ and $B^{(k+1)} = B^{(k)} / \Soc(B^{(k)})$ ($k \geq 0$). This process may terminate in a skew brace with $1$ element, which gives rise to the notion of \emph{right-nilpotency degree}: here, we say that a skew brace is \emph{right-nilpotent} of \emph{degree $\leq k$} for some integer $k \geq 0$, if $B^{(k)} = 0$. If $B$ is right-nilpotent of degree $\leq k$ but not of degree $k-1$, we say $B$ is \emph{right-nilpotent} of \emph{degree $k$}.

Skew braces of right-nilpotency degree $\leq 2$ can be constructed in a particularly easy way:

\begin{pro} \label{pro:right_nilpotent_of_degree_2_construction}
    Let $B$ be a skew brace. Then the following statements are equivalent:

    \begin{enumerate}
        \item $B$ is right-nilpotent of degree $\leq 2$.
        \item The map $\lambda: (B,+) \to \Aut(B,+)$; $g \mapsto \lambda_g$ is a homomorphism of groups.
        \item $\lambda_{\lambda_a(b)} = \lambda_b$ is satisfied for all $a,b \in B$.
    \end{enumerate}
\end{pro}

\begin{proof}
    \cite[Theorem 3.13]{bi_skew_brace_blocks}.
\end{proof}

The following proposition shows that the conditions imposed on $\alpha$ in the previous proposition are sufficient for the construction of a skew brace: 

\begin{pro} \label{pro:construct_skew_braces_with_alpha}
    Given a group $(B,+)$ and a homomorphism $\alpha: (B,+) \to \Aut(B,+)$; $a \mapsto \alpha_a$ with
    \begin{equation} \label{eq:invariance_for_nilp_deg_2}
        \alpha_{\alpha_a(b)} = \alpha_b,
    \end{equation}
    then $B_{\alpha} = (B,+,\subcirc{\alpha})$ is a skew brace, where
    \begin{equation} \label{eq:constructing_nilp_deg_2}
        a \subcirc{\alpha} b = a + \alpha_a(b).
    \end{equation}
\end{pro}

\begin{proof}
    This follows from a straightforward calculation.
\end{proof}

Observe that, by \cref{pro:right_nilpotent_of_degree_2_construction}, such a skew brace is necessarily right-nilpotent of degree $\leq 2$.

Finally, we need to recall the correspondence between regular subgroups of the holomorph and skew braces:

Given a group $G$, define for an element $g \in G$, the \emph{left translation} as the map $l_g \in \Sym_G$ that is given by $l_g(x) = gx$.
It is well-known that the group of left translations, $L_G = \{ l_g : g \in G \}$ is a subgroup of $\Sym_G$ that is isomorphic to $G$. The \emph{holomorph} of $G$ is now defined as the normalizer $\Hol(G) = N_{\Sym_G}(L_G) \leq \Sym_G$. It is well-known that $\Hol(G)$ admits a factorization $\Hol(G) = L_G \rtimes \Aut(G)$, where $\Aut(G)$ is the automorphism group of $G$, considered as a a subgroup of $\Sym_G$.

The following result connects skew braces and regular subgroups of the holomorph of a group:

\begin{thm} \label{thm:correspondence_regular_subgroups_and_skew_braces}
    Let $G = (G,+)$ be a group, denoted additively. Then the following two assignments are the mutually inverse constituents of a bijective correspondence between the set of skew brace structures $(G,+,\circ)$ and the set of regular subgroups $H \leq \Hol(G,+)$:

    \begin{enumerate}
        \item To a regular subgroup $H \leq \Hol(G,+)$, assign the skew brace structure $(G,+,\subcirc{H})$ by $\pi(e) \subcirc{H} h = \pi(h)$ ($\pi \in H$).
        \item To a skew brace structure $(G,+,\circ)$, assign the regular subgroup $L_{(H,\circ)} \leq \Hol(G,+)$.
    \end{enumerate}
\end{thm}

\begin{proof}
    \cite[Theorem 4.2]{Guarnieri_Vendramin}.
\end{proof}

From the formula $g \circ h = g + \lambda_g(h)$ it follows that $L_{(G,\circ)} \leq L_{(G,+)} \rtimes \mathrm{im}(\lambda)$. On the other hand, we also see that if $H \leq L_{(G,+)} \rtimes A$ for a subgroup $A \leq \Aut(G,+)$, the $\lambda$-map of $(G,+,\subcirc{H})$ has its image in $A$. We conclude:

\begin{pro} \label{pro:restrict_lambda}
    Let $G = (G,+)$ be a group and let $A \leq \Aut(G,+)$. Then the above correspondence restricts to a bijective correspondence between:
    \begin{enumerate}
        \item regular subgroups of $G \rtimes A$, and
        \item skew brace structures on $(G,+)$ with $\mathrm{im}(\lambda) \leq A$.
    \end{enumerate}
\end{pro}

Observe that, in particular, the trivial brace structure $(G,+,+)$ corresponds to the regular subgroup $L_G \leq \Hol(G,+)$.

\section{Order automorphisms of left \texorpdfstring{$\ell$}{l}-groups} \label{sec:order_automorphisms}

For our investigation of groups of $\IG$-type, it will be favourable to gain a good understanding of the automorphisms of the underlying lattice. In this section, we will show that this is indeed possible under certain \emph{rigidity} conditions. It will be shown in \cref{sec:IG_for_AT} that spherical Artin-Tits groups indeed satisfy these rigidity conditions.

\begin{defn} \label{defn:rigid_garside_group}
    Let $G$ be a left $\ell$-group. We say that $G$ is \emph{rigid} if $G$ is noetherian and the following two conditions are satisfied:
    \begin{enumerate}
        \item For any $x,y \in \At(G)$ such that $x \neq y$, there is a unique $z \in \At(G)$ such that
        $xz \leq x \vee y$.
        \item For any $x \in \At(G)$, there is at most one $z \in \At(G)$ such that $xz \not\leq x \vee y$ for all $y \in \At(G)$.
    \end{enumerate}

    If $G^{\mathrm{op}}$ is rigid, $G$ is called \emph{dually rigid}. If $G$ is rigid and dually rigid, then $G$ is called \emph{bi-rigid}. 
\end{defn}

Using left-invariance, the following proposition follows immediately from the definition:

\begin{pro} \label{pro:shifting_rigidity}
    Let $G$ be a rigid left $\ell$-group. Then
    \begin{enumerate}
        \item For any $g,h_1,h_2 \in G$, with $h_1 \neq h_2$ and $g \prec h_1,h_2$, there is a unique $h' \in G$ such that $h' \succ g$ and $h' \leq h_1 \vee h_2$.
        \item For any $g,h \in G$ with $g \prec h$, there is at most one $f \in G$ such that $f \succ h$ and $f \not\leq h \vee h'$ for any $h' \in G$ with $h' \succ g$.
    \end{enumerate}
\end{pro}

It turns out that rigidity for a left $\ell$-group implies that the group is also rigid under order-automorphisms:

\begin{pro} \label{pro:rigidity_of_positive_cone}
    Let $G$ be a rigid left $\ell$-group and let $\varphi \in \Aut(G,\leq)$ be an order-automorphism such that $\varphi(x) = x$ for all $x \in \At(G)$. Then $\varphi(g) = g$ for all $g \in G^+$.
\end{pro}

\begin{proof}
    Note that $G$ is trivial if $|\At(G)| = 0$, so we may assume that $\At(G)$ is non-empty.

    We prove that for $\varphi$ as in the statement of the proposition, we have $\varphi(g) = g$ for all $g \in G^+$ by induction over the height $H = H(g)$. Note that the case $H=1$ is the statement of the lemma. For $H = 0$, then $g = e$. In the case that $|\At(G)| = 1$, $e$ is the unique element covered by $g$, so $\varphi(e) = e$. If $|\At(G)| > 1$, then there are $x,y \in \At(G)$ with $x \neq y$. For these atoms, we have $x \wedge y = e$. Consequently, $\varphi(e) = \varphi(x) \wedge \varphi(y) = x \wedge y = e$.

    So assume now that we are given a $g \in G^+$ with $H = H(g) \geq 2$, and that all $h \in G^+$ with $H(h) < H$ are fixed under $\varphi$. Now pick a successor chain $e = h_0 \prec h_1 \prec \ldots \prec h_H = g $ with $h_0 = e$. First suppose that there is an $h' \in G^+$ with $h_{H-2} \prec h'$ and $g \leq h' \vee h_{H-1}$. By our inductive assumption, $h_{H-2}$, $h_{H-1}$ and $h'$ are fixed under $\varphi$. Applying $\varphi$, we get that $\varphi(g) \succ \varphi(h_{H-1}) = h_{H-1}$ and $\varphi(g) \leq \varphi(h_{H-1}) \vee \varphi(h') = h_{H-1} \vee h'$. But by \cref{pro:shifting_rigidity}, this implies $\varphi(g) = g$. If there is \emph{no} $h' \succ h_{H-2}$ with $g \leq h_{H-1} \vee h'$, then applying $\varphi$, this implies the non-existence of an $h' \succ h_{H-2}$ with $\varphi(g) \leq h_{H-1} \vee h'$. Again, \cref{pro:shifting_rigidity} implies that $\varphi(g) = g$, thus finishing the inductive step.
    \end{proof}

\begin{lem} \label{cor:shifting_rigidity_of_negative_cones}
Let $G$ be a dually rigid left $\ell$-group and let $g \in G$. If $\varphi \in \Aut(G,\leq)$ is such that $\varphi(h) = h$ for all $h \in G$ with $h \prec g$, then $\varphi(h) = h$ for all $h \in g^{\downarrow} = \{ f \in G: f \leq g \}$.
\end{lem}

\begin{proof}
    Dualizing \cref{pro:rigidity_of_positive_cone}, we see that the statement is true when $g = e$. Else, consider the map $\varphi': G \to G$; $h \mapsto g^{-1} \varphi(gh)$. The map $\varphi'$ is an automorphism of ordered sets such that $\varphi(h) = h$ for all $h \in \Atd(G)$. It follows that $\varphi'(h) = h$ for all $h \in G^-$. As a consequence, $\varphi(h) = h$ for all $h \in g^{\downarrow}$.
\end{proof}

\begin{pro} \label{pro:rigidity_of_birigid_l_groups}
    Let $G$ be a bi-rigid left $\ell$-group where $\bigwedge \Atd(G)$ exists, and let $\varphi \in \Aut(G,\leq)$ be an order-automorphism such that $\varphi(x) = x$ for all $x \in \At(G)$, then $\varphi = \id_G$.
\end{pro}

\begin{proof}
    By \cref{pro:rigidity_of_positive_cone}, $\varphi$ fixes $G^+$ pointwise. Let $s = \bigwedge \Atd(G)$, then $s^{-1} \geq e$.

    Now let $g \in G^-$ be arbitrary and let $h = g \vee e \in G^+$. In particular, $hs^{-1} \in G^+$ and each $h' \prec hs^{-1}$ is of the form $h' = hs^{-1}x$ with $x \in \Atd(G)$. As $s \leq x$ for all $x \in \Atd(G)$ it follows that $s^{-1}x \geq e$ for all $x \in \Atd(G)$. Therefore, $h' \in G^+$ for all $h' \prec hs^{-1}$ which implies that $\varphi(h') = h'$ for all $h' \prec hs^{-1}$. By \cref{cor:shifting_rigidity_of_negative_cones}, we infer that $\varphi(h') = h'$ for all $h' \in (hs^{-1})^{\downarrow}$ and as $g \leq h \leq hs^{-1}$, it follows that $\varphi(g) = g$.
\end{proof}

\begin{thm} \label{thm:rigidity_of_restriction_to_atoms}
    Let $G$ be a bi-rigid left $\ell$-group where $\bigwedge \Atd(G)$ exists, then the restriction
    \[
    \rho: \Aut(G,\leq)_e \to \Sym_{\At(G)}; \ \varphi \mapsto \varphi |_{\At(G)}
    \]
    is injective.
\end{thm}

\begin{proof}
    We show that $\ker(\rho)$ is trivial: if $\rho(\varphi) = \id$, then this means nothing else than $\varphi(x) = x$ for all $x \in \At(G)$. By \cref{pro:rigidity_of_birigid_l_groups}, it follows that $\varphi = \id_G$.
\end{proof}

We now give the definition of a group of \emph{$\IG$-type}:

\begin{defn} \label{defn:IGamma}
    Let $G$ be a left-ordered group. A left-ordered group $H$ is of \emph{$\IG$-type}, if there is an order-isomorphism $\iota: (H,\leq) \overset{\sim}{\to} (G,\leq)$ with $\iota(e_H) = e_G$. We call the tuple $(H,\iota)$ an \emph{$\IG$-formation}.

    For a given left-ordered group $G$, two $\IG$-formations $(H,\iota)$, $(H',\iota')$ are \emph{equivalent} if there is an isomorphism of left-ordered group $f: H \to H'$ such that $\iota' \circ f = \iota$.

    An $\IG$-formation $(H,\iota)$ is called \emph{trivial} if it is equivalent to the $\IG$-formation $(G,\mathrm{id}_G)$, that is, if $\iota$ is an isomorphism of left-ordered groups.
\end{defn}

Obviously, if there is \emph{any} order-isomorphism $\iota: H \overset{\sim}{\to} G$, then there is one with $\iota(e_H) = e_G$, by left-invariance. However, fixing $\iota(e_H) = e_G$ once and for all, will later spare us from shifting around order-isomorphisms.

Note that a left-ordered group $G$ is equivalent to a partially ordered set $(P,\leq)$ with a distinguished point $e$ and a regular action of a group $G$ on $(P,\leq)$. We therefore obtain for the group $\Aut(G,\leq)$ of order-automorphisms - that are not necessarily group automorphisms - the following decomposition:

\begin{pro} \label{pro:factorizing_automorphism_groups}
    Let $G$ be a left-ordered group. Then, then group $\Aut(G,\leq)$ factorizes as
    \[
    \Aut(G,\leq) = L_G \cdot \Aut(G,\leq)_e \ ; \ L_G \cap \Aut(G,\leq)_e = \{ \mathrm{id}_G \},
    \]
    where $\Aut(G,\leq)_e = \{ \varphi \in \Aut(G,\leq) : \varphi(e) = e \}$.
\end{pro}

\begin{proof}
    It is well-known that a regular subgroup of a permutation group gives rise to such a factorization, and $L_G \leq \Aut(G,\leq)$ is a regular subgroup.
\end{proof}

\begin{pro} \label{pro:correspondence_regular_subgroups_IG_formations}
    Let $G$ be a left-ordered group. Then the following two assignments are the mutually inverse constituents of a bijective correspondence between equivalence classes of $\IG$-formations and regular subgroups $H \leq \Aut(G,\leq)$:
    \begin{enumerate}
        \item To an $\IG$-formation $(H,\iota)$, assign the regular subgroup ${}^{\iota}L_H \leq \Aut(G,\leq)$.
        \item To a regular subgroup $H \leq \Aut(G,\leq)$, assign the $\IG$-formation $(H',\iota_H)$ where $H' = (H,\leq)$ is the left-ordered group with $\pi \leq \rho \Leftrightarrow \pi(e_G) \leq \rho(e_G)$, and $\iota_H(\pi) = \pi(e_G)$.  
    \end{enumerate}
\end{pro}

\begin{proof}
    We prove that the first assignment is well-defined: given a left-ordered group $H$, we see that $L_H \leq \Aut(H,\leq)$ by definition, and regularity is obvious. Now $\iota$ is an isomorphism between ordered sets, so ${}^{\iota}L_H$ is indeed a subgroup of $\Aut(G,\leq)$, and it is regular as $\iota$ is bijective. Given two $\IG$-formations $(H_i,\iota_i)$ ($i = 1,2$), that are equivalent via the isomorphism $f: H_1 \overset{\sim}{\to} H_2$, we obtain that ${}^{\iota_1}L_{H_1} = {}^{\iota_2 \circ f}( {}^{f^{-1}}L_{H_2} ) = {}^{\iota_2}L_{H_2}$, therefore the constructed regular subgroup is independent of the choice of a representative.

    Now given a regular subgroup $H \leq \Aut(G,\leq)$, we need to show that $\pi \leq \rho \Leftrightarrow \pi(e) \leq \rho(e)$ ($\pi,\rho \in H$) indeed defines a left-ordered group. But this is clear as $\rho_1(e) \leq \rho_2(e)$ implies $(\pi \circ \rho_1)(e) \leq (\pi \circ \rho_2)(e)$ for $\pi,\rho_1,\rho_2 \in \Aut(G,\leq)$, simply by the definition of an order-automorphism. Furthermore, the mapping $\iota_H: H' \to G$; $\pi \mapsto \pi(e_G)$ is an isomorphism of ordered sets by the definition of $H'$.

    We are left with proving that these assignments are bijective: first, let $(H,\iota)$ be an $\IG$-formation. We have to show that $({}^{\iota}L_H,\iota')$ with $\iota'(\pi) = \pi(e)$ is equivalent to $(H,\iota)$: it is clear that $f: H \to {}^{\iota}L_H$; $h \mapsto {}^{\iota}l_h$ is an isomorphism of groups. On the other hand, for $h,h' \in H$ we have the chain of equivalences:
    \[
    h \leq h' \Leftrightarrow \iota(h) \leq \iota(h') \Leftrightarrow ({}^{\iota}l_h)(e_G) \leq ({}^{\iota}l_{h'})(e_G) \Leftrightarrow f(h) \leq f(h').
    \]
    Furthermore, $(\iota' \circ f)(h) = {}^{\iota}l_h(e_G) = \iota(h)$, which proves equivalence.
    
    On the other hand, if $H \leq \Aut(G,\leq)$ is a regular subgroup, we only have to show that with the map $\iota_H: H \to G$; $\pi \mapsto \pi(e)$, we have ${}^{\iota_H}L_H = H$. But this is easily checked: for $g \in G$, $\pi \in H$, pick $\rho \in H$ with $\iota_H(\rho) = g$. With this choice, we have
    \[
    ({}^{\iota_H}l_{\pi})(g) = \iota_H (\pi \circ \iota_H^{-1}(g)) =  \iota_H (\pi \circ \rho) = (\pi \circ \rho)(e) = \pi(g).
    \]
\end{proof}

We will now prove that under reasonable obstructions, a bi-rigid left $\ell$-group admits only finitely many $\IG$-formations.

\begin{thm} \label{thm:finiteness_theorem}
    Let $G$ be a bi-rigid left $\ell$-group such that $\At(G)$ is finite. Then up to equivalence, there are only finitely many $\IG$-formations.
\end{thm}

\begin{proof}
    Note that $\bigwedge \Atd(G)$ exists as $|\Atd(G)|= |\At(G)| < \infty$. As $\Sym_{\At(G)}$ is finite, \cref{thm:rigidity_of_restriction_to_atoms} shows that $\Aut(G,\leq)_e$ is finite. Furthermore, by \cref{pro:noetherian_l_groups_generated_by_atoms}, $G$ is finitely generated. As $\Aut(G,\leq) = L_G \cdot \Aut(G,\leq)_e$ (\cref{pro:factorizing_automorphism_groups}), we see that $(\Aut(G,\leq):L_G) = d < \infty$ where $d = |\Aut(G,\leq)_e|$. As $L_G \cong G$ is finitely generated, it follows that $\Aut(G,\leq)$ is finitely generated.

    By \cref{pro:correspondence_regular_subgroups_IG_formations}, equivalence classes of $\IG$-formations are in bijective correspondence with regular subgroups of $\Aut(G,\leq)$. As $\Aut(G,\leq)$ is finitely generated, and each regular subgroup is of index $d$ in $\Aut(G,\leq)$, it follows that there can only be finitely many equivalence classes of $\IG$-formations.
\end{proof}

\section{Groups of \texorpdfstring{$\IGam$}{IGam}-type for spherical Artin-Tits groups} \label{sec:IG_for_AT}

We now solve the original problem of Dehornoy et al. by providing a characterization of $\IGam$-formations whenever $\Gamma$ is a spherical Artin-Tits group. In order to achieve this, we make use of the rigidity of spherical Artin-Tits groups. We start by determining their order-automorphisms:

\begin{pro} \label{bi-rigidity_of_at_groups}
    Let $\Gamma$ be an Artin-Tits group of spherical type and let $\varphi \in \Aut(\Gamma,\leq)$ be an order-automorphism with $\varphi(e) = e$. Then $\varphi$ is a diagram automorphism.
\end{pro}

\begin{proof}
    By \cref{thm:brieskorn_saito}, $\Gamma$ is a noetherian left $\ell$-group with $\At(G) = \{ \sigma_i : i \in S \}$.

    We prove that $\Gamma$ is rigid: let $i,j \in S$, then thanks to \cref{thm:brieskorn_saito}, $\sigma_i \vee \sigma_j \in \Gamma$ can be determined with respect to $\Gamma^+$. It is clear that $\sigma_i \vee \sigma_j \leq r_{m_{ij}}(\sigma_i,\sigma_j) =: r$. As $r_{m_{ij}}(\sigma_i,\sigma_j)$ and $r_{m_{ij}}(\sigma_j,\sigma_i)$ are the only expressions representing $r$ in $\Gamma^+$, it follows that
    \begin{align}
     \{ r_k(\sigma_i,\sigma_j) : 0 \leq k \leq m_{ij} \} \cup \{ r_k(\sigma_j,\sigma_i) : 0 \leq k \leq m_{ij} \} & = [e,r], \label{eq:interval_e_r} \\
     \{ r_k(\sigma_i,\sigma_j) : 0 \leq k \leq m_{ij} \} \cap \{ r_k(\sigma_j,\sigma_i) : 0 \leq k \leq m_{ij} \} & = \{e,r \} \label{eq:interval_cut} 
    \end{align}
    We observe that any $g \in [e,r] \setminus \{ r \}$ has a unique expression as a positive word in the generators $\sigma_i,\sigma_j$, which implies that either $g \not\geq \sigma_i$ or $g \not\geq \sigma_j$. Therefore, $\sigma_i \vee \sigma_j = r$.
    
    Furthermore, this observation, together with \eqref{eq:interval_e_r} shows that $x = \sigma_j$ is the unique $x \in \At(G)$ with $\sigma_ix \leq \sigma_i \vee \sigma_j$. Letting $j$ vary through $S$, we also observe that, given $i \in S$, the atom $x = \sigma_i$ is unique with the property that $\sigma_ix \not\leq \sigma_i \vee \sigma_j$ for any $j \in S$. Therefore, $\Gamma$ is a rigid left $\ell$-group.

    Due to the symmetry of the relations defining $\Gamma$, it follows that $\Gamma$ is also dually rigid, so $\Gamma$ is bi-rigid. 

    By \cref{thm:rigidity_of_restriction_to_atoms}, it follows that each order-automorphism $\varphi: \Gamma \to \Gamma$ with $\varphi(e)$ is determined by its restriction to $\At(\Gamma) = \{ \sigma_i : i \in S \}$. Let $\varphi$ be such an order-automorphism, then there is a $\phi \in \Sym_n$ with $\varphi(\sigma_i) = \sigma_{\phi(i)}$ ($i \in S$).

    Note that $|[e,\sigma_i \vee \sigma_j]| = 2m_{ij}$ for all $i,j \in S$, which follows from \eqref{eq:interval_e_r} and \eqref{eq:interval_cut} and the discussion thereafter, so
    \[
    2m_{\phi(i)\phi(j)} = |[e,\sigma_{\phi(i)} \vee \sigma_{\phi(j)}]| = |[e,\varphi(\sigma_i) \vee \varphi(\sigma_j)]| = |[e,\sigma_i \vee \sigma_j]| = 2m_{ij}.
    \]
    Therefore, $\delta_{\phi}$ is a diagram automorphism of $\Gamma$ and also, an automorphism of the ordered set $(\Gamma,\leq)$ (\cref{pro:diagram_automorphisms_are_order_automorphisms}).
    As $\varphi(\sigma_i) = \delta_{\phi}(\sigma_i)$ for all $i \in S$, it follows from \cref{thm:rigidity_of_restriction_to_atoms} that $\varphi = \delta_{\phi}$, so $\varphi$ is a diagram automorphism.
\end{proof}

\begin{thm} \label{thm:order_automorphism_of_at_groups}
    Let $\Gamma$ be an Artin-Tits group of spherical type and let $\mathfrak{D}_{\Gamma}$ be the group of diagram automorphisms of $G$, then
    \[
    \Aut(\Gamma,\leq) = L_{\Gamma} \rtimes \mathfrak{D}_{\Gamma} \leq \Hol(\Gamma).
    \]
\end{thm}

\begin{proof}
    It follows from \cref{pro:factorizing_automorphism_groups} that there is a factorization $\Aut(\Gamma,\leq) = L_{\Gamma} \cdot \Aut(\Gamma,\leq)_e$. By \cref{thm:order_automorphism_of_at_groups}, $\Aut(\Gamma,\leq)_e = \mathfrak{D}_{\Gamma}$, which normalizes $L_{\Gamma}$, therefore this factorization is a semidirect product.
\end{proof}

\begin{thm} \label{thm:IGamma_groups_are_skew_braces}
    Let $(\Gamma,+)$ be an Artin-Tits group of spherical type, written additively. Then the following two assignments are the mutually inverse constituents of a bijective correspondence between the equivalence classes of $\IGam$-formations and skew brace structures $(\Gamma,+,\circ)$ with $\mathrm{im}(\lambda) \leq \mathfrak{D}_{\Gamma}$:

    \begin{enumerate}
        \item To an $\IGam$-formation $(H,\iota)$, assign the skew brace $(\Gamma,+,\subcirc{\iota})$ where
        \[
        g \circ' h = \iota(\iota^{-1}(g) \circ \iota^{-1}(h))
        \]
        \item To a skew brace structure $(\Gamma,+,\circ)$ with $\mathrm{im}(\lambda) \leq \mathfrak{D}_{\Gamma}$, assign the $\IGam$-formation given by $((\Gamma,\circ),\mathrm{id}_{\Gamma})$ where $(\Gamma,\circ)$ inherits its order from $(\Gamma,+)$.
    \end{enumerate}
\end{thm}

\begin{proof}
    By \cref{pro:correspondence_regular_subgroups_IG_formations}, each equivalence class of $\IGam$-formations corresponds to a regular subgroup of $\Aut(\Gamma,\leq)$. By \cref{thm:order_automorphism_of_at_groups}, $\Aut(\Gamma,\leq) = L_{\Gamma} \rtimes \mathfrak{D}_{\Gamma} \leq \Hol(\Gamma,+)$, so the correspondence is established by combining \cref{pro:restrict_lambda} and \cref{pro:correspondence_regular_subgroups_IG_formations}.
\end{proof}

By means of \cref{thm:IGamma_groups_are_skew_braces}, the classification problem for $\IGam$-structures for a spherical Artin-Tits group $\Gamma$ translates to the problem of classifying skew brace structures on $\Gamma$ whose $\lambda$-maps are diagram automorphisms. For $\Gamma = \Z^n$, this problem is equivalent to the classification of finite involutive non-degenerate set-theoretic solutions to the Yang--Baxter equation \cite{GIVdB_IType,JO_IType}, so a full solution of the classification problem for general spherical Artin-Tits groups is unlikely. However, in the following, we will see that a full classification of $\IGam$-formations is possible when $\Gamma$ is an \emph{irreducible} Artin-Tits group!

\begin{pro} \label{pro:at_groups_with_non_trivial_IGam_groups}
    Let $\Gamma$ be an \emph{irreducible} spherical Artin-Tits group. If there exist non-trivial $\IGam$-formations, then $\Gamma$ is of type $A_n$ ($n \geq 2$), $D_n$ ($n \geq 4$), $E_6$, $F_4$ or $I_n$ ($n \geq 4$).
\end{pro}

\begin{proof}
    By \cref{thm:order_automorphism_of_at_groups}, we see that if $\mathfrak{D}_{\Gamma}$ is trivial, then $\Aut(\Gamma,\leq) = L_{\Gamma}$ and each $\IGam$-formation is equivalent to $(\Gamma,\mathrm{id}_{\Gamma})$. Only for the types listed in the statement of the proposition, $\mathfrak{D}_{\Gamma}$ is non-trivial, which is quickly checked by an inspection of the Coxeter-Dynkin diagrams for irreducible spherical Artin-Tits groups (see \cite{Humphreys_Reflection}, for example).
\end{proof}

In the following, we call a Coxeter-Dynkin diagram \emph{oddly laced} if each label is either $2$ or an odd integer. Furthermore, we will call an Artin-Tits group oddly laced if its Coxeter-Dynkin diagram is. Given this notion, we now prove the following lemma:

\begin{lem} \label{lem:oddly_laced_to_abelian}
    Let $\Gamma$ be an oddly laced, irreducible Artin-Tits group with standard generators $\sigma_i$ ($i \in S$) and let $A$ be an abelian group. If $f: \Gamma \to A$ is a group homomorphism, then there is a fixed $a \in A$ such that $f(\sigma_i) = a$ for all $i \in S$.
\end{lem}

\begin{proof}
    As the Coxeter-Dynkin diagram of $\Gamma$ is connected, it is sufficient to show that $f(\sigma_i) = f(\sigma_j)$ whenever $m = m_{ij}$ is odd. In this case, we have
    \begin{align*}
    f(r_m(\sigma_i,\sigma_j)) & = f(r_m(\sigma_j,\sigma_i)) \\
    \Rightarrow r_m(f(\sigma_i),f(\sigma_j)) & = r_m(f(\sigma_j),f(\sigma_i)) \\
    \Rightarrow f(\sigma_i)^{\frac{m+1}{2}}f(\sigma_j)^{\frac{m-1}{2}} & = f(\sigma_i)^{\frac{m-1}{2}}f(\sigma_j)^{\frac{m+1}{2}} \\
    \Rightarrow f(\sigma_i) & = f(\sigma_j).
    \end{align*}
\end{proof}

\begin{pro}  \label{pro:construction_of_irreducible_IGamma_groups}
    If $\Gamma = (\Gamma, +)$ is an \emph{irreducible} Artin-Tits group, written additively, then each $\IGam$-formation is equivalent to exactly one formation of the form $(\Gamma_{\alpha},\iota)$. Here $\Gamma_{\alpha} = (\Gamma,\subcirc{\alpha})$ with $g \subcirc{\alpha} h = g + \alpha_g(h)$ and $\iota(g) = g$, where $\alpha: (\Gamma,+) \to \mathfrak{D}_{\Gamma}$; $g \mapsto \alpha_g$ is a homomorphism satisfying
    \begin{equation} \label{eq:invariance_for_AT_skew_braces}
    \alpha_h = \alpha_{\alpha_g(h)} \ (g,h \in \Gamma).
    \end{equation}
\end{pro}

\begin{proof}
    Let an irreducible spherical Artin-Tits groups $(\Gamma,+)$ be given. \cref{thm:IGamma_groups_are_skew_braces} shows that in order to find all $\IGam$-formations up to equivalence, we need to characterize the skew brace structures $\Gamma = (\Gamma,+,\circ)$ where $\mathrm{im}(\lambda) \leq \mathfrak{D}_{\Gamma}$. Note first that, by \cref{pro:construct_skew_braces_with_alpha}, each valid choice of $\alpha$ will result in a skew brace structure.
    
    An observation of Coxeter-Dynkin diagrams on page \pageref{tab:dynkin} shows that $|\mathfrak{D}_{\Gamma}| \leq 3$, except when $\Gamma$ is of type $D_4$. Therefore, if $\Gamma$ is an irreducible spherical Artin-Tits group and not of type $D_4$, then $(\Gamma:\Soc(\Gamma)) \leq 3$, which implies that $\Gamma/\Soc(\Gamma)$ is trivial in all of these cases and $\Gamma$ has right nilpotency degree $\leq 2$. Then, the desired representation for $(\Gamma,\circ)$ follows from \cref{pro:right_nilpotent_of_degree_2_construction}.

    We only need to pay special attention to the case when $(\Gamma,+)$ is of type $D_4$ and the image of the $\lambda$-map is the whole of $\mathfrak{D}_{\Gamma} \cong \Sym_3$. In this case $\Gamma^{(1)} = \Gamma/\Soc(\Gamma)$ is of size $6$, so either $(\Gamma^{(1)},+) \cong \Z_6$ or $(\Gamma^{(1)},+) \cong \Sym_3$.

    In the first case, we apply the fact that $(\Gamma,+)$ is oddly laced, together with \cref{lem:oddly_laced_to_abelian} to the factor map $\pi: (\Gamma,+) \twoheadrightarrow (\Gamma^{(1)},+)$, in order to show that $\Bar{\sigma_i} = \Bar{\sigma_j}$ holds in $\Gamma^{(1)}$ for all $i,j \in S$. In particular, $\overline{\delta(\sigma_i)} = \overline{\sigma_i}$ for all $\delta \in \mathfrak{D}_{\Gamma}$, $i \in S$. This implies that $\overline{\lambda_g(\sigma_i)} = \overline{\sigma_i}$ for all $g \in \Gamma$, $i \in S$. Therefore, $\Gamma^{(1)}$ is a trivial skew brace, so $\Gamma$ is of right nilpotency degree $\leq 2$ and \cref{pro:right_nilpotent_of_degree_2_construction} applies.

    In the case when $(\Gamma^{(1)},+) \cong \Sym_3$, let $\varepsilon: (\Gamma,+) \twoheadrightarrow \Sym_3$ be a homomorphism inducing this isomorphism. As $\varepsilon$ is surjective, there is an $i \in S$ with $\mathrm{sgn}(\varepsilon(\sigma_i)) = -1$, where $\mathrm{sgn}: \Sym_3 \to \{ \pm 1 \}$ is the sign homomorphism. Now \cref{lem:oddly_laced_to_abelian} applied to the homomorphism $\mathrm{sgn} \circ \varepsilon: (\Gamma,+) \to \{ \pm 1 \}$ shows that $\mathrm{sgn}(\varepsilon(\sigma_i)) = -1$ holds for all $i \in S$, so $\varepsilon$ maps all generators $\sigma_i$ to transpositions. As the generators $\sigma_i$ ($1 \leq i \leq 3$) pairwise commute (see the labelling on page \pageref{tab:dynkin}), their images $\varepsilon(\sigma_i)$ ($1 \leq i \leq 3$) are pairwise commuting transpositions, which in $\Sym_3$ implies that $\varepsilon$ maps $\sigma_1,\sigma_2,\sigma_3$ to the same element of $\Sym_3$. Therefore, $\Bar{\sigma_1} = \Bar{\sigma_2} = \Bar{\sigma_3}$ in $(\Gamma^{(1)},+)$. As each $\delta \in \mathfrak{D}_{\Gamma}$ permutes $\{ \sigma_i : 1 \leq i \leq 3 \}$ and fixes $\sigma_4$, it follows that $\overline{\delta(\sigma_i)} = \overline{\sigma_i}$ holds for all $i \in S$ and $\delta \in \mathfrak{D}_{\Gamma}$, therefore we have $\overline{\lambda_g(\sigma_i)} = \overline{\sigma_i}$ in $\Gamma^{(1)}$ for $g \in \Gamma$, $i \in S$. Again, $\Gamma^{(1)}$ is a trivial skew brace, so $\Gamma$ is of right-nilpotency degree $\leq 2$ and the desired representation follows again from \cref{pro:right_nilpotent_of_degree_2_construction}.
    
\end{proof}

We can now classify all $\IGam$-formations where $\Gamma$ is an irreducible, spherical Artin-Tits group. In the following, we will only mention the non-trivial ones. By \cref{pro:at_groups_with_non_trivial_IGam_groups}, non-trivial $\IGam$-formations only exist if $\Gamma$ is of type $A_n$ ($n \geq 2$), $D_n$ ($n \geq 4$), $E_6$, $F_4$ or $I_n$ ($n \geq 4$).

We first discuss the cases when the Coxeter-Dynkin diagram of $\Gamma$ is \emph{oddly laced}. Furthermore, suppose for now that $\Gamma$ is not of type $D_4$. In these cases, $\mathfrak{D}_{\Gamma}$ is abelian, so by \cref{lem:oddly_laced_to_abelian}, for each homomorphism $\alpha: \Gamma \to \mathfrak{D}_{\Gamma}$, there is a fixed $\id_{\Gamma} \neq \delta \in \mathfrak{D}_{\Gamma}$ such that $\alpha_{\sigma_i} = \delta$ for all $i \in S$. This is compatible with the relations of $\Gamma$, and the invariance condition \eqref{eq:invariance_for_AT_skew_braces} is clearly satisfied.

If $\Gamma$ is of type $I_n$ and $\alpha: \Gamma \to \mathfrak{D}_{\Gamma}$ is non-trivial, then one generator $\sigma_i$ satisfies $\alpha_{\sigma_i} = \delta_{(12)}$. We may assume without restriction that this is $\sigma_1$. Then \cref{eq:invariance_for_AT_skew_braces} shows that $\alpha_{\sigma_2} = \alpha_{\alpha_{\sigma_1}(\sigma_2)} = \alpha_{\sigma_1} = \delta_{(12)}$. So $\alpha_{\sigma_1} = \alpha_{\sigma_2} = \delta_{(12)}$ which is an assignment compatible with the relations in $\Gamma$.

We now consider the case when $\Gamma$ is of type $F_4$. In this case, $\mathfrak{D}_{\Gamma}$ is abelian and as $m_{12} = 3$ (see the labellings on page \pageref{tab:dynkin}), we see that for each homomorphism $\alpha: \Gamma \to \mathfrak{D}_{\Gamma}$, there is a $\delta \in \mathfrak{D}_{\Gamma}$ such that $\alpha_{\sigma_1} = \alpha_{\sigma_2} = \delta$. As we suppose that $\IGam$ is a non-trivial skew brace, there is a $g \in \Gamma$ with $\alpha_g = \delta_{(14)(23)}$, and the invariance condition \eqref{eq:invariance_for_AT_skew_braces} now implies that $\alpha_{\sigma_3} = \alpha_{\alpha_g(\sigma_1)} = \alpha_{\sigma_3} = \delta$ and similarly, $\alpha_{\sigma_2} = \alpha_{\alpha_g(\sigma_4)} = \alpha_{\sigma_4} = \delta$. Therefore, $\alpha_{\sigma_i} = \delta$ for $1 \leq i \leq 4$, which is compatible with the relations of $\Gamma$ and satisfies condition \eqref{eq:invariance_for_AT_skew_braces}.

Consider last the case when $\Gamma$ is of type $D_4$. If the Coxeter-Dynkin diagram of $\Gamma$ is oddly laced, it follows from \cref{lem:oddly_laced_to_abelian} that the homomorphisms $\alpha: \Gamma \to \mathfrak{D}_{\Gamma}$ with \emph{abelian} image in $\mathfrak{D}_{\Gamma}$ that satisfy \cref{eq:invariance_for_AT_skew_braces}, are exactly those that are given by $\alpha_{\sigma_i} = \delta$ ($1 \leq i \leq 4$) with some fixed $\mathrm{id}_{\Gamma} \neq \delta \in \mathfrak{D}_{\Gamma}$. If $\mathrm{im}(\alpha) = \mathfrak{D}_{\Gamma}$, then an argument similar to the one in the proof of \cref{pro:construction_of_irreducible_IGamma_groups} shows that all surjective homomorphisms $\alpha: \Gamma \twoheadrightarrow \mathfrak{D}_{\Gamma}$ map all $\sigma_i$ ($1 \leq i \leq 4$) to transpositions. Furthermore, there is a fixed assignment of values to $a,b,c$ such that $\{ a,b,c \} = \{ 1,2,3\}$ and
$\alpha_{\sigma_i} = \delta_{(a\ b)}$ for $1 \leq i \leq 3$ and, as $\alpha$ is surjective, $\alpha_{\sigma_4} = \delta_{(b\ c)}$. Note that this choice is indeed compatible with the relations in $\Gamma$!

We can now summarize our findings:

\begin{thm} \label{thm:classification_irreducible_IGam_formations}
    Each \emph{non-trivial} $\IGam$-formation $(H,\iota)$, where $\Gamma = (\Gamma,+)$ is an irreducible spherical Artin-Tits group, is equivalent to \emph{exactly one} $\IGam$-formation of the form $(\Gamma_{\alpha},\iota)$ with $\Gamma_{\alpha} = (\Gamma, \subcirc{\alpha})$ and $\iota = \mathrm{id}_{\Gamma}$, where $g \subcirc{\alpha} h = g + \alpha_g(h)$, with a homomorphism $\alpha: (\Gamma,+) \to \mathfrak{D}_{\Gamma}$ from the following table:

    \begin{center}
        \begin{tabular}{c|c}
   Type  & $\alpha$  \\ \hline
    $A_n$ ($n \geq 3$) & $\sigma_i \mapsto \delta_{(1\ n)(2\ n-1)\ldots}$ ($1 \leq i \leq n$)\\
    $D_n$ ($n \geq 4$) & $\sigma_i \mapsto \delta_{(1\ 2)}$ ($1 \leq i \leq n$) \\
    $D_4$ & $\sigma_i \mapsto \delta_{(a\ 3)}$ ($1 \leq i \leq 4$) with $a \in \{ 1,2 \}$ \\
    $D_4$ & $\sigma_i \mapsto \delta_{(1\ a\ b)}$ ($1 \leq i \leq 4$) with $\{ a,b \} = \{2,3 \}$ \\
    $D_4$ & $\sigma_i \mapsto \delta_{(a\ b)}$ ($1 \leq i \leq 3$), $\sigma_4 \mapsto \delta_{(b\ c)}$ with $\{a,b,c \} = \{1,2,3 \}$ \\
    $E_6$ & $\sigma_i \mapsto \delta_{(1\ 5)(2\ 4)}$ ($1 \leq i \leq 6$) \\
    $F_4$ & $\sigma_i \mapsto \delta_{(1\ 4)(2\ 3)}$ ($1 \leq i \leq 4$) \\
    $I_n$ ($n \geq 4$) & $\sigma_i \mapsto \delta_{(1\ 2)}$ ($1 \leq i \leq 2$)
\end{tabular}
    \end{center}
    
\end{thm}

\begin{rem}
    Let $\Gamma = (\Gamma,+)$ be an arbitrary spherical Artin-Tits group and let $(H,\iota)$ be an $\IGam$-formation. By \cref{thm:IGamma_groups_are_skew_braces}, this formation corresponds to a skew brace structure $(\Gamma,+,\circ)$ with $\mathrm{im} (\lambda) \leq \mathfrak{D}_{\Gamma}$. It is well-known that for an Artin-Tits group $\Gamma$, the element $\Delta = \bigvee \At(\Gamma)$ is a Garside element (see \cite[§5]{ABS_Coxeter}). Furthermore, $\Delta$ is fixed under $\mathfrak{D}_{\Gamma}$, so $\lambda_g(\Delta) = \Delta$ for all $g \in \Gamma$. Now pick $k > 0$ such that $\Tilde{\Delta} = k \Delta \in \Zen(\Gamma,+) \cap \Soc(\Gamma)$, then $\Tilde{\Delta} \in \Gamma^+$ and for all $g \in \Gamma$, we have
    \[
    \Tilde{\Delta} \circ g = \Tilde{\Delta} + g = g + \Tilde{\Delta} = g \circ \lambda_{g}^{-1}(\Tilde{\Delta}) = g \circ \lambda_{g}^{-1}(k \cdot \Delta) = g \circ (k \cdot \Delta) = g \circ \Tilde{\Delta}.
    \]
    so $\Tilde{\Delta} \in \Zen(\Gamma,\circ)$ which implies $\Tilde{\Delta} \circ \Gamma^+ = \Gamma^+ \circ \Tilde{\Delta}$. In particular, the right- and left-divisors of $\Tilde{\Delta}$ in $(\Gamma^+,\circ)$ coincide. Furthermore, as $k \cdot \Tilde{\Delta} = \Tilde{\Delta}^k$ for all integers $k$, it follows that $\Gamma^+$ is a Garside monoid for $(\Gamma,\circ)$ with Garside element $\Tilde{\Delta}$. Therefore, groups of $\IGam$-type are Garside groups whenever $\Gamma$ is a spherical Artin-Tits group.
\end{rem}

\begin{rem}
    If $\Gamma$ is of type $I_n$ ($n \geq 2$), we have seen in \cref{thm:classification_irreducible_IGam_formations} that there is exactly one non-trivial $\IGam$-formation $(H,\iota)$. Considering the interval $[e,\Delta]_{\Gamma}$ where $\Delta = r_n(\sigma_1,\sigma_2) = r_n(\sigma_2,\sigma_1)$, one can check that $\Tilde{\Delta} = \iota^{-1}(\Delta)$ is central in $H$ and an argument similar to the one in the previous remark shows that $\Tilde{\Delta}$ is a Garside element for $H$. Reading off the relations of $H$ from $[e,\Tilde{\Delta}]$, one obtains that $H \cong \genrel{a,b}{a^n = b^n}_{\mathrm{gr}}$, which is the \emph{torus-type} group $T_{n,n}$. \cite[Example I.2.7]{Foundations_Garside}.
\end{rem}

\begin{rem}
    Note that any spherical Artin-Tits group $\Gamma$ decomposes as a direct product $\Gamma = \prod_{i \in I} \Gamma_i$ where each $\Gamma_i$ is an irreducible Artin-Tits group. This decomposition corresponds to the decomposition of the Coxeter-Dynkin diagram into connected components. Identifying $i \sim j$ whenever $\Gamma_i$ and $\Gamma_j$ are of the same type, one obtains an equivalence relation on $I$. Note that $i \sim j$ if and only if the Coxeter-Dynkin diagrams of $\Gamma_i$ and $\Gamma_j$ are equivalent. Now putting $\mathcal{I} = I/\!\!\sim$, one obtains a coarser decomposition
    \[
    \Gamma = \prod_{J \in \mathcal{I}} \Gamma_J, \textnormal{ where} \quad \Gamma_J = \prod_{i \in J} \Gamma_i.
    \]
    We see that $\Gamma_J \unlhd \Gamma$, as $\Gamma_J$ is a direct factor. Furthermore, each diagram automorphism leaves unions of equivalent connected components setwise invariant, so all $\Gamma_J$ are invariant subgroups under $\mathfrak{D}_{\Gamma}$.

    Given a skew brace structure $(\Gamma,+,\circ)$ corresponding to an $\IGam$-formation $(H,\iota)$ via \cref{thm:IGamma_groups_are_skew_braces}, it follows that the subgroups $\Gamma_J$ ($J \in \mathcal{I}$) are strong left ideals of the skew brace structure on $\Gamma$. As $(\Gamma,+)$ is a direct product of the components $\Gamma_J$, one obtains that $(\Gamma,\circ)$ is a matched product of the permutable subgroups $(\Gamma_J,\circ)$ ($J \in \mathcal{I}$), so $H$ is a matched product of the permutable subgroups $\iota^{-1}(\Gamma_J)$ ($J \in \mathcal{I}$).
\end{rem}

\bibliography{references}

\begin{thebibliography}{10}

\bibitem{bjorner_combinatorics}
A.~Bjorner and F.~Brenti.
\newblock {\em Combinatorics of {C}oxeter {G}roups}.
\newblock Graduate Texts in Mathematics. Springer Berlin Heidelberg, 2009.

\bibitem{ABS_Coxeter}
E.~Brieskorn and K.~Saito.
\newblock Artin-{G}ruppen und {C}oxeter-{G}ruppen.
\newblock {\em Invent. math.}, 17:245--271, 1972.

\bibitem{Chouraqui_Garside}
F.~Chouraqui.
\newblock {G}arside groups and {Y}ang–{B}axter equation.
\newblock {\em Comm. Algebra}, 38(12):4441--4460, 2010.

\bibitem{Dehornoy_Garside_groups}
P.~Dehornoy.
\newblock Groupes de {G}arside.
\newblock {\em Annales Scientifiques de l’École Normale Supérieure}, 35(2):267--306, 2002.

\bibitem{Foundations_Garside}
P.~Dehornoy, F.~Digne, E.~Godelle, D.~Krammer, and J.~Michel.
\newblock {\em Foundations of {G}arside {T}heory}.
\newblock EMS tracts in mathematics. European Mathematical Society, 2015.

\bibitem{Garside_braid}
F.~A. Garside.
\newblock {The braid group and other groups}.
\newblock {\em Q. J. Math.}, 20(1):235--254, 01 1969.

\bibitem{GIVdB_IType}
T.~Gateva-Ivanova and M.~Van~den Bergh.
\newblock Semigroups of {I}-type.
\newblock {\em J. Algebra}, 206:97--112, 1998.

\bibitem{goffa_jespers}
I.~Goffa and E.~Jespers.
\newblock Monoids of {IG}-type and maximal orders.
\newblock {\em Journal of Algebra}, 308(1):44--62, 2007.

\bibitem{Guarnieri_Vendramin}
L.~Guarnieri and L.~Vendramin.
\newblock Skew braces and the {Y}ang-{B}axter equation.
\newblock {\em Math. Comp.}, 86(307):2519--2534, 2017.

\bibitem{Humphreys_Reflection}
J.~Humphreys.
\newblock {\em Reflection {G}roups and {C}oxeter {G}roups}.
\newblock Cambridge Studies in Advanced Mathematics. Cambridge University Press, 1990.

\bibitem{JO_IType}
E.~Jespers and J.~Okni\'nski.
\newblock Monoids and groups of {$I$}-type.
\newblock {\em Algebr. Represent. Theory}, 8(5):709--729, 2005.

\bibitem{Rump-Geometric-Garside}
W.~Rump.
\newblock Right $\ell$-groups, geometric {G}arside groups, and solutions of the quantum {Y}ang–{B}axter equation.
\newblock {\em J. Algebra}, 439, 2015.

\bibitem{bi_skew_brace_blocks}
L.~Stefanello and S.~Trappeniers.
\newblock On bi-skew braces and brace blocks.
\newblock {\em Journal of Pure and Applied Algebra}, 227(5):107295, 2023.

\end{thebibliography}
\bibliographystyle{abbrv}

\end{document}